\documentclass[11pt]{article}
\usepackage{amsmath,amssymb, amsfonts, amsthm, latexsym, hyperref,titlesec }
\usepackage{bbm,xcolor,graphicx}

\titleformat{\subsubsection}[runin] {\normalfont\bfseries}{\thesubsubsection}{0.7em}{\addperiod}
\newcommand{\addperiod}[1]{#1.}

\titlespacing\subsubsection{5pt}{4pt plus 4pt minus 2pt}{5pt plus 2pt minus 2pt}

\newtheorem{theorem}{Theorem}

\newtheorem{question}[theorem]{Question}

\newcommand{\E}{\mathbb{E}}

\newcommand{\eps}{\varepsilon}
\newcommand{\ol}{\overline}

\DeclareMathOperator*{\esssup}{ess\,sup}

\title{A density version for  H\"aggstr\"om's theorem}
\author{Itai Benjamini \and Ori Gurel-Gurevich}

\begin{document}

\maketitle
\begin{abstract}
Given invariant percolation on a regular tree, where the probability of an edge to be open equals $p$,
is it always possible to find an infinite self-avoiding path along which the density of open edges is bigger then $p$?
\end{abstract}

Let $S$ be an invariant percolation on the edges of the $d$-regular tree, where the probability of an edge being open equals $p$. We think of $S$ as an invariant process with values in $\{0,1\}$ (1 corresponds to open edges). For $\ol{x}=(x_0,x_1,x_2,\dots)$ an infinite self-avoiding path, let $D(\ol{x})$ be the density of the percolation along $\ol{x}$, that is,
$$D(\ol{x})=\limsup_{n\to\infty} \frac{1}{n}\sum_{k=1}^n S(x_{k-1},x_k)$$
and let
$$D(S)=\sup_{\ol{x}} D(\ol{x}) \ .$$

In general, this is a random variable that is $F_\infty$-measurable, where $F_\infty$ is the tail $\sigma$-algebra. We may look at the essential supremum of this random variable and define
$$D_d(p)=\inf_S \esssup D(S) \ ,$$
where the infimum is taken over all invariant percolation distributions on the $d$-regular tree. For background on invariant percolation, see e.g. \cite{BLPS}.


Obviously, $D_d(p)$ is monotone in $p$ and $D_d(p)\ge p$.

\begin{question}
Is $D_d(p)>p$ for any $d\ge3$ and $0<p<1$?
\end{question}

More generally we may ask
\begin{question}
What is $D_d(p)$?
\end{question}

In his seminal paper \cite{H}, Olle H\"aggstr\"om proved that any invariant percolation on the $d$-regular tree, with marginal at least $\frac{2}{d}$, has an infinite cluster. In particular, we get that $D_d(\frac{2}{d})=1$ and specifically $D_3(\frac23)=1$.

\begin{theorem}
$D_3\left(1-\frac{1}{\sqrt{3}}\right)\ge \frac12$.
\end{theorem}
\begin{proof}
Take two iid samples from the percolation distribution and look at their maximum. If $p\ge 1-\frac{1}{\sqrt{3}}$ then this new percolation has marginal $\ge\frac23$, so by H\"aggstr\"om's theorem  there is an infinite cluster a.s. and in particular there is $\ol{x}$ with all the edges open. At least one of the two original percolations must have $D(\ol{x})\ge \frac12$, so $D(S)\ge \frac12$
\end{proof}

More generally, define
$$a(d,k)=1-\sqrt[k]{1-\frac{2}{d}} .$$

\begin{theorem}
$D_d(a(d,k))\ge \frac{1}{k}$
\end{theorem}
\begin{proof}
The same proof as the previous theorem, except that you take $k$ copies and work on the $d$-regular tree.
\end{proof}

Notice that for $d=3$ and $2\le k \le 5$ we have $a(d,k)<\frac{1}{k}$, so we get that $D_3(p)>p$ for any $p\in [a(3,k),\frac1{k})$, but if $k\ge 6$ then $a(3,k)>\frac{1}{k}$, so we obtain no new information.

However, for $d\ge 4$ we have $a(d,k)<\frac{1}{k}$ for all $k$, so we get some that $D_d(p)>p$ for any $p\in \cup_{k=1}^\infty [a(d,k),\frac1{k})$.

In fact,
\begin{theorem}
For any $d\ge 4$ and any $0<p<1$ we have $D_d(p)>p$.
\end{theorem}
\begin{proof}
All we need to do is show that for $d\ge 4$ we have $\cup_{k=1}^\infty [a(d,k),\frac1{k})=(0,1)$. We claim that for any $d\ge 4$ and any $k\ge 1$ we have $a(d,k)\le \frac{1}{k+1}$ which means that these intervals are overlapping.

Now
$$1-\sqrt[k]{1-\frac{2}{d}}\le \frac{1}{k+1}$$
is equivalent to
$$\left(1-\frac{1}{k+1}\right)^k \le 1-\frac{2}{d}$$
and the left hand side is decreasing (as a function of $k$) so the maximum is obtained for $k=1$ and it is $\frac12\le 1-\frac{2}{d}$.
\end{proof}

\begin{theorem}
For any $d$, the function $D_d$ is uniformly continuous.
\end{theorem}
\begin{proof}
Fix $d$. Let $B$ be bernoulli percolation on the $d$-regular tree with marginal $\eps$. For a path of length $n$ the probability of getting at least $a$ 1's is bounded by
$${n \choose an} \eps^{an} \le (2 \eps^a)^n .$$

Since there are $d(d-1)^{n-1}$ paths of length $n$ we get that when $a>\frac{\log(2(d-1))}{\log(1/\eps)}$ the probability of a path with $an$ 1's decays exponentially. We conclude that
$$D(B)\le f_d(\eps) := \frac{\log(2(d-1))}{\log(1/\eps)} .$$

We now claim that if $0\le p < q \le 1$ and $q-p\le \eps$ then $D_d(q)-D_d(p)\le f_d(3\eps)$. This implies uniform continuity since $f_d(\eps)\to 0$ when $\eps\to 0$.

To show the claim, let $S$ be an invariant percolation with marginal $p$ and $B$ bernoulli percolation with marginal $3\eps$. Let $S'$ be their maximum. Then $S'$ has marginal $p+(1-p)3\eps \ge q$ (since we may assume that $p\le 2/3$, for $p > 2/3$ we have $D_d(p)=D_d(q)=1$). Therefore, with positive probability, there is an infinite path $\ol{x}$ such that the density of $s'$ along $\ol{x}$ is at least $D_d(q)$. But the contribution of $B$ to the density of $\ol{x}$ is at most $f_d(3\eps)$, so the density of $S$ along $\ol{x}$ is at least $D_d(q)-f_d(3\eps)$.
\end{proof}

In particular, $D_d(p)\to 1$ as $p\to\frac23$ so for some $a<\frac23$ we have $D_3(p) > p$ for all $p\in [a,1)$. However, we still don't know that $D_3(p)>p$ for all $0<p<1$ and specifically that $D_3(1/2)>(1/2)$.

When $d\to\infty$ we have that if $p=\frac{1}{d}+\frac{1}{d^2}$ we have $1-(1-p)^2= 2p - p^2 > \frac{2}{d}$ so again we have $D(p)\ge \frac12$. This works for any fixed $k$, so if we define the limit
$$D_\infty(x)=\lim_{d\to\infty} D\left(\frac{2}{d}x\right)$$
we know that $D_\infty(x)\ge \frac{1}{k}$ for any $x>\frac{1}{k}$.

\begin{question}
Is it true that $D_\infty(x)=x$?
\end{question}

%


Note that the same methods apply to site percolation on regular trees. However, in that case, as $d\to\infty$ the threshold in H\"aggstr\"om's theorem tends to $\frac12$ rather then 0. Indeed, the tree is a bipartite graph and the partition into two sides is invariant, hence we can define a percolation that choose one of the sides with equal probabilities and then put 1s on this side and 0s on the other. This gives a marginal of $\frac12$ and also density of $\frac12$ along any self-avoiding path. This percolation is ergodic, but have a nontrivial tail $\sigma$-algebra.

\begin{question}
What can be said about site percolation on regular trees if we require that the tail $\sigma$-algebra is trivial?
\end{question}

We may also consider more general processes, i.e. not $\{0,1\}$-valued.

\begin{question}
Is it true that for any invariant, non-constant process $S$ on the edges of a regular tree, $D(S)>\E[S(e)]$, where $e$ is some/any edge of the tree?
\end{question}

An interesting side question is this:
\begin{question}
Is it true that when you replace the $\limsup$ in the definition of $D(\ol{x})$ by $\liminf$ you get the same function? If not, do our result still hold for the $\liminf$ version?
\end{question}

\noindent {\bf Remark:} H\"aggstr\"om's theorem was extended to nonamenable Cayley graphs \cite{BLPS}, all the discussion above adapts to this set up.


\noindent
{\bf Acknowledgements:} thanks to Tom Hutchcroft for a useful comment.

\end{document}